\documentclass[
hf,
]{ceurart}

\usepackage{amsthm} 
\newcommand\seq\vdash 
\newcommand\limp{\mathbin{\rightarrow}} 
\usepackage{prooftree} 
\usepackage{tikz}
\usepackage{graphicx}
\usepackage{rotating}

\newtheorem{prop}{Proposition}
\newcommand\N{\mathbb{N}}
\newcommand\A{\mathcal{A}}
\newcommand\F{\mathcal{F}}
\newcommand\xor{\mathrel{|}}

\newcommand\Hfun{\mathsf{Hfun}}
\newcommand\Hpred{\mathsf{Hpred}}

\usepackage{listings}
\begin{document}

\copyrightyear{2024}
\copyrightclause{Copyright for this paper by its authors.
  Use permitted under Creative Commons License Attribution 4.0
  International (CC BY 4.0).}

\conference{ARQNL, Nancy,  2024}

\title{A proof-theoretical approach to\\ some extensions of first order quantification}

\author[1]{Loïc Allègre}[%
email=loic.allegre@lirmm.fr,
]
\address[1]{LIRMM (Univ Montpellier \& CNRS), Montpellier, France}
\author[2]{Ophélie Lacroix}[%
email=ophelie.lacroix@resolve.tech, 
url=https://sites.google.com/site/ophelielacroixnlp/,
]
\address[2]{Resolve, Copenhagen, Danemark}
\author[1]{Christian Retoré}[%
orcid=0000-0002-2401-9158,
email=christian.retore@lirmm.fr,
url=http://www.lirmm.fr/~retore,
]


\begin{abstract}
Generalised quantifiers, which include Henkin's branching quantifiers, have been introduced by Mostowski and Lindström and developed as a substantial topic application of logic, especially model theory, to linguistics with work by Barwise, Cooper, Keenan.   

In this paper, we mainly study the proof theory  of some non-standard quantifiers as second order formulae . Our first example is the usual pair of first order quantifiers (for all / there exists) when individuals are viewed as individual concepts handled by second order deductive rules. Our second example is the study of a second order translation  of the simplest branching quantifier: ``A member of each team and a member of each board of directors know each other", for which we propose a second order treatment.  
\end{abstract}



\maketitle

\section{Generalisation of usual first order quantification}

Common first order quantifiers $\exists$ and $\forall$ have been formalised the way they are in standard mathematics by Frege \cite{Frege:1879} whose logical and philosophical view  matches Hilbert desiderata regarding logical foundations of mathematics.\cite{KK86,Goldfarb1979,HBvol1,HBvol2} 

By that time, mathematicians and logicians were making little  difference  between the interpretations of quantifiers in \emph{the} standard model and their proof rules --- before the work of Skolem or Gödel  mathematicians and logicians made little distinction 
between syntax and semantics, they worked with an interpreted language, see e.g. \cite{KK86} —— perhaps Hilbert was more demanding regarding quantifiers because he focused on foundations of mathematics \cite{Goldfarb1979}.

Extension of usual quantification have mainly been considered for modelling faithfully quantification modes that one founds in ordinary language, like numbers (``three students came to the party"), ``most" (``most students came to the party"),  percentage (``a third of the students came to the party"), vague quantifiers (``few students came to the party") etc. 

This gave rise to the theory of generalised quantifiers, initially intended for model theory \cite{Mostowski1957,Lindstrom1966} that was intensively developed in connection with linguistics.\cite{BarwiseCooper1981,Keenan1986,sep-generalized-quantifiers}.  In such a setting, the lexical item expressing a quantifier (say ``most A are B") is viewed as a function with two arguments that are predicates. This fits in well with the logico-functional view known as Montague semantics\cite{montague74b}.
Most (!) generalised quantifiers can be viewed as a second order construction. For instance the interpretation of a generalised quantifier depending on two unary predicates like ``most" is interpreted as a second order construction, i.e. is true whenever the pairs of unary predicates it is applied to is a ``legal" pair of subsets of the domain, namely a pair of sets $(A, B)$ such that $|A\cap B|>|A\setminus B|$.

The paper is organised as follows. We first provide a reminder on second order logic, because this topic is not so common. Next, we present a second order view of usual first order quantification as second order quantification over individual concepts --- and show the two formulations are proved to be equivalent provided the individual concepts are standard (i.e. they may not be empty, as opposed to some Kripke and Muskens variants). Then, after a quick presentation of generalised quantifiers,  we study  the simplest branching quantifier as a second order construction, and we propose direct rules for this quantifier.

\section{A reminder on second order logic}
\label{2ndOrder} 

We briefly remind the reader with basic facts about second order logic,  following \cite[Chapter 5]{vanDalen2013}, and one may also refer to the survey \cite{sep-logic-higher-order}.  

\subsection{Language} A second-order language is based on a first-order language (the first three items in the list below), and extended with an infinitely enumerable set of predicate variables, each of them endowed with an arity (the last item in this list).
\begin{itemize} 
\item an infinite enumerable set of first order (a.k.a. individual) variables $x_i$, with $i$ in $\N$
\item a set of constants $c_i$, with $i$ in $I$ ($I$ is often enumerable, but this is not required);\footnote{The first order language may also include an enumerable set of first order functions, but an $n$-ary function $g$ in the language can replaced with an $n+1$-ary predicate $G$ with an axiom that $F$ is a function, and at second order there is a formula $F[X]$ saying that the $n+1$-ary predicate $X$ corresponds to an $n$-ary function.}
\item an enumerable (or finite) set of predicate constants $P_i^n$ with $i$ in $\N$ each of them with an arity $n$ (as in a first order language) -- a predicate constant with arity $0$ is a proposition; 
\item an enumerable set of predicate variables $X_i^n$ $i$ in $\N$ each of them with an arity $n$ -- a predicate variable  with arity $0$ is a propositional variable. 
\end{itemize} 

Second order formulae  are defined ``as expected": an atomic formula is $Z^n(u_1,\ldots,u_n)$
with $Z$ a predicate constant or a predicate variable of arity $n$ and the $u_i$ being $n$ first order terms (here first order variables or first order constants since we have no functions). Les us call $\A$ the set of atomic formulae . Then the set of second order formulae  $\F$ is defined by 

\begin{center} 
$\F\ {=::}\ \A \xor \lnot \F \xor \F \land \F \xor \F  \lor \F \xor \F \limp \F 
\xor \forall x \F 
\xor \exists x \F 
\xor \forall^n X_i^n  \F 
\xor \exists^n X_i^n  \F$

\it 
where $x$ stands for an individual variable while $X_i^n$ stands for a predicate variable of arity $n$. 
\end{center} 

The formula  $A \Leftrightarrow B$ is just a short-hand for $(A \limp B) \land (B \limp A)$. 

Although we shall not always write the $^n$ superscript in $\forall^n$  and $\exists^n$ beware that there are different pairs of second order quantifiers ($\exists^n/\forall^n$), one pair for each arity. The occurrences of the variable $x_i$ or $X_i^n$ are bound by the closest $\forall x, \exists x, \forall^n X_i^n, \exists^n X_i^n$ --- if any --- above them in the formula tree. 

\subsection{Proof rules in natural deduction} \label{proofrules} 

We use natural deduction with standard rules as can be found in \cite{vanDalen2013}. As we limit ourselves to classical logic, an extra principle is needed: \textsl{tertium non datur} ($A\lor\lnot A$ for all $A$) or \textsl{reductio ad absurdum} (from a deduction with conclusion $\bot$ under hypothesis $\lnot A$, conclude $A$), see e.g. \cite{MootRetore2016classical}

The proof rules for second order quantifiers, namely the introduction and elimination rules of $\forall^n$ and $\exists^n$ are as expected, they are similar to the rules for first order quantifiers, \textsl{mutatis mutandi}: 

$$
\begin{array}{ccc}  
\begin{prooftree} 
\[
\leadsto 
\forall^n X_i^n\ T[X_i^n] 
\] 
\justifies 
T[X_{i,k}^n(t^1_k,\ldots,t^n_k):=\phi^n(t^1_k,\ldots,t^n_k)]
\using (\forall^n)_E
\end{prooftree}
&&
\begin{prooftree} 
\[
\leadsto 
T[X_i^n]
\] 
\justifies 
\forall^n X_i^n\ T[X_i^n] 
\using (\forall^n)_I
\end{prooftree} 
\\[5em]  
\begin{prooftree} 
\exists^n X_i^n\ T[X_i^n] 
\[ 
\bigg[T[X_i^n]\bigg]^k 
\leadsto 
\psi 
\]
\justifies 
\psi 
\using (\exists^n)_E^k
\end{prooftree} 
&& 
\begin{prooftree} 
\[
\leadsto 
T[X_{i,k}^n(t^1_k,\ldots,t^n_k):=\phi^n(t^1_k,\ldots,t^n_k)]
\] 
\justifies 
\exists^n X_i^n T[X_i^n]
\using (\exists^n)_I
\end{prooftree} 
\end{array}
$$

where 
\begin{enumerate}
\item 
$T[X_i^n]$ stands for a formula in which the predicate variable $X_i^n$ may occur (but that is not mandatory, as for first order quantification). 
\item 
There should be no free occurrence of $X_i^n$ in the hypotheses of the introduction rule $(\forall^n X_i^n)_I$ nor in the elimination rule $(\exists X_i^n)_E$ --- as in the first order $\forall_I$ and $\exists_E$ introduction rules. 
\item 
The obscure notation\footnote{This point if often under explained in the literature.} $T[X_{i,k}^n(t^1_k,\ldots,t^n_k):=\phi^n(t^1_k,\ldots,t^n_k)]$ requires some explanation. This formula  
stands for the formula obtained by replacing 
\begin{itemize} 
\item 
the $k^{\mathrm{th}}$ occurrence $X_{i,k}^n$ of $X_i^n$ which is applied to $n$ terms $(t^1_k,\ldots,t^n_k)$ 
\item with a formula with $n$ free variables applied to the very same terms $(t^1_k,\ldots,t^n_k)$ 
\end{itemize}
with the requirement that no originally free variable in $(t^1_k,\ldots,t^n_k)$ becomes bound after the application of $\phi$ to $(t^1_k,\ldots,t^n_k)$. 
\newline\indent Here is an example: let $\phi(x,y)=P(x,y)\land Q(y,a)$, let $T[X_1^2]=X_{1,1}^2(z,a)\land X_{1,2}^2(a,b)$ -- mind the second subscript of $X_{1,\bullet}^2$ which indicates the occurrence number (there are two occurrences of the predicate variable $X_1^2$ in $T[X_1^2]$. Then $T[X_{1,k}^2:=\phi(t^1_k,\ldots,t^n_k)]$ is 
$(\phi(x,y)=P(x,y)\land Q(y,a))[x:=z;y:=a] \land (\phi(x,y)=P(x,y)\land Q(y,a))[x:=a;y:=b]$
that is $(P(z,a)\land Q(a,a))\land (P(a,b)\land Q(b,a))$. 
From the definition and the example, it is unsurprising that second order unification,  is undecidable \cite{Goldfarb1981}. 
\item 
In the rule $(\exists^n)_E$ the expression $\bigg[T[X_i^n]\bigg]_k$ indicates that the hypothesis 
$T[X_i^n]$ has been cancelled during the $(\exists^n)_E$ number $k$. 
\end{enumerate}

Some remarks: 

\begin{enumerate} 
\item 
    this proof system can derive the comprehension axiom: 
$$\exists X^n \forall x_1 \ldots x_n\left[\varphi\left(x_1, \ldots, x_n\right) \leftrightarrow X^n\left(x_1, \ldots, x_n\right)\right]$$ 
\item 
equality can be defined \`a  la Leibnitz: 
$x=y : \forall^1 X^1 X^1(x) \limp X^1(y)$ (because of negation there is no need to use $\Leftrightarrow$ in this definition)
\item \label{Ex} being equal to $x$ is a property $E_x(y):\forall^1 X^1(x)\limp X^1(y)$. 
\item 
the Dedekind finiteness,  ``any injective function is surjective" is definable: 

$\forall^2 X^2\quad 
((\forall x \forall y \forall z 
(X(x,y) \land X(x,z) \limp y=z))
\land 
(\forall x \forall y \forall z 
(X(y,x)\land X(z,x) \limp y=z)))$ 

\hspace*{2cm}
$  \limp  
(\forall w \exists u\  X(u,w))$
\end{enumerate} 

Finally, using implication $\limp$, first order $\forall$, and propositional second order $\forall^0$ one can define false, $\bot$, negation $\lnot$, the propositional connective $\land$, $\lor$, first order existential quantification $\exists$.  Adding  $\forall^n$  to $\limp$, $\forall$, one can also define $\exists^n$. 
Thus, the expressive power of second order propositional quantifier $\forall^0$ is impressive. 

\subsection{Standard and non-standard models, completeness} 

We here follow \cite{Henkin1950,sep-logic-higher-order,vanDalen2013}. 

A second order model consists in a first order model, i.e. with a  domain $D$, endowed with a set of sets of tuples of length $n$ for each $n\in\N$ in order to interpret predicate variables of arity $n$: they may vary in a fixed subset $A^n$ of $\mathcal{P}(D^n)$ which is not necessarily the full powerset $\mathcal{P}(D^n)$; for this structure to define a model, it must enjoy 
the comprehension axiom scheme  $\exists^n X^n \forall x_1\cdots \forall  x_n 
[\phi(x_1,\ldots,x_n) \leftrightarrow X^n(x_1,\ldots,x_n)]
$ where the $n$-ary predicate variable $X^n$ 
does not appear in $\phi$ --- in other words 
the subsets of $A^n$ must include the interpretations of the formulae  with $n$ free variables. The comprehension axiom scheme is derivable from the existential introduction rule given above. 

When for any arity $n$ this subset $A^n$ is $\mathcal{P}(D^n)$ the model is said to be \emph{standard} (or full). Standard  models, which satisfy the comprehension scheme, match the easiest intuition: a predicate variable of arity $n$ varies in all possible subsets of $D^n$. As second order logic can express the finiteness of the model, 
completeness and compactness do not hold with standard models. 

Second order logic may be encoded in first order logic: predicate variables $X_i^n$ are viewed as individual constants interpreted in a domain (that also contains standard individuals), and some additional predicate constants of arity $n+1$ are needed to mimic the application of  a predicate variable of arity $n$  to $n$ terms. Then a second order formula is provable within second order logic whenever its first order translation is provable in first order logic. So applying first order completeness theorem one gets that a second order formula is provable in second order logic if and only if it is true in all second order models (including the non standard ones).  As a consequence of completeness, compactness holds, so one can have a model with at least $n$ elements for each integer $n$, which is  Dedekind finite.

\section[A second order view of first order quantification: quantifying over individual concepts]{A second order view of first order quantification:\\ quantifying over individual concepts} \label{individualconcepts}

\subsection{Individual concepts}
Individual concepts view an individual as a formula $\phi[x]$ with a single free variable $x$  such that there is a single individual satisfying the formula $\phi[x]$.\footnote{This notion of concept is somehow related to concepts in description logics, and the individual concepts that we use correspond to individual names cf. e.g. \cite{Rudolph2011DL}.}
This can be said in second order logic: a formula $\phi[x]$ with a single free variable $x$ is said to be an individual concept whenever there is at most one individual $x$ such that  $\phi[x]$ and at least one $x$ such that $\phi[x]$ --- so it makes exactly one $x$ such that $phi[x]$. That the concept $\phi$ is an \emph{individual concept} can actually be expressed in second order logic:  
\centerline{$C(\phi) : (\forall x \forall y \phi(x) \land \phi(y) \limp x=y) \land \exists z \phi(z) $} 

Individual concepts are close to Montague semantics and Leibnitz identity
(where an individual is identified with the set of all properties it enjoys) \cite{Kripke1971,Kripke1972,montague74b}.
For reasons like possible worlds semantics, see e.g. the discussion in \cite[chapter 4]{Fitting1998foml}, some logicians consider a variant of individual concepts that are possibly empty.   This may look strange but if you think individual concepts are some kind of proper name that are part of the logical language, it is hard to tell what a proper name refers to before the reference actually exists. For instance, the individual concept $\mathrm{Gödel}(x)$ has no reference in Egyptian times. Hence Kripke in the 70s dropped the existence condition from individual concepts (see \cite{Muskens2012} for a recent account of those ideas), thus obtaining a formula with a lower logical complexity profile: 
\centerline{$C(\phi) : (\forall x \forall y \phi(x) \land \phi(y) \limp x=y)$}

\subsection[First order universal quantification as universal quantification over non-empty individual concepts]{First order universal quantification as universal quantification over non-empty individual concepts}

The simplest quantifiers one can try to view as a second order construction are clearly the usual first order quantifiers $\forall,\exists$. So let us compare the second order quantification over individual concepts to usual quantification. 

\begin{prop} \label{ICQuniv} First order universal quantification and second order quantification over individual concepts are equivalent: 
\begin{enumerate} 
\item given a property $\varphi(X)$ of individual concepts,  the following equivalence holds: $\forall x \varphi^{\downarrow}(x) \Leftrightarrow \forall X(C(X) \limp  \varphi(X))$ where $\varphi^{\downarrow}(x):=\exists X(C(X) \wedge$ $X(x) \wedge \varphi(X))$.
\item given a property  $\psi(x)$ of individuals, the following equivalence holds:  $\forall x \psi(x) \Leftrightarrow \forall X\left(C(X) \limp  \psi^{\uparrow}(X)\right)$ where  $\psi^{\uparrow}(X):=\exists x(X(x) \wedge \psi(x))$. 
\end{enumerate} 
\end{prop}

\begin{proof}
Let us first observe that there is a simple formal proof without assumption that $C(E_x)$ i.e. that ``being equal to $x$" (cf. section \ref{2ndOrder} item \ref{Ex}) is an individual concept, $E_x(y): y=x : \forall^1 X^1\ (X^1(x)\limp X^1(y))$,  and let us call this proof $\delta$, because we are going to use it several times:

\begin{center} 
$\delta:$\hspace{2cm}\ 
\nopagebreak 

\includegraphics[scale=0.3]{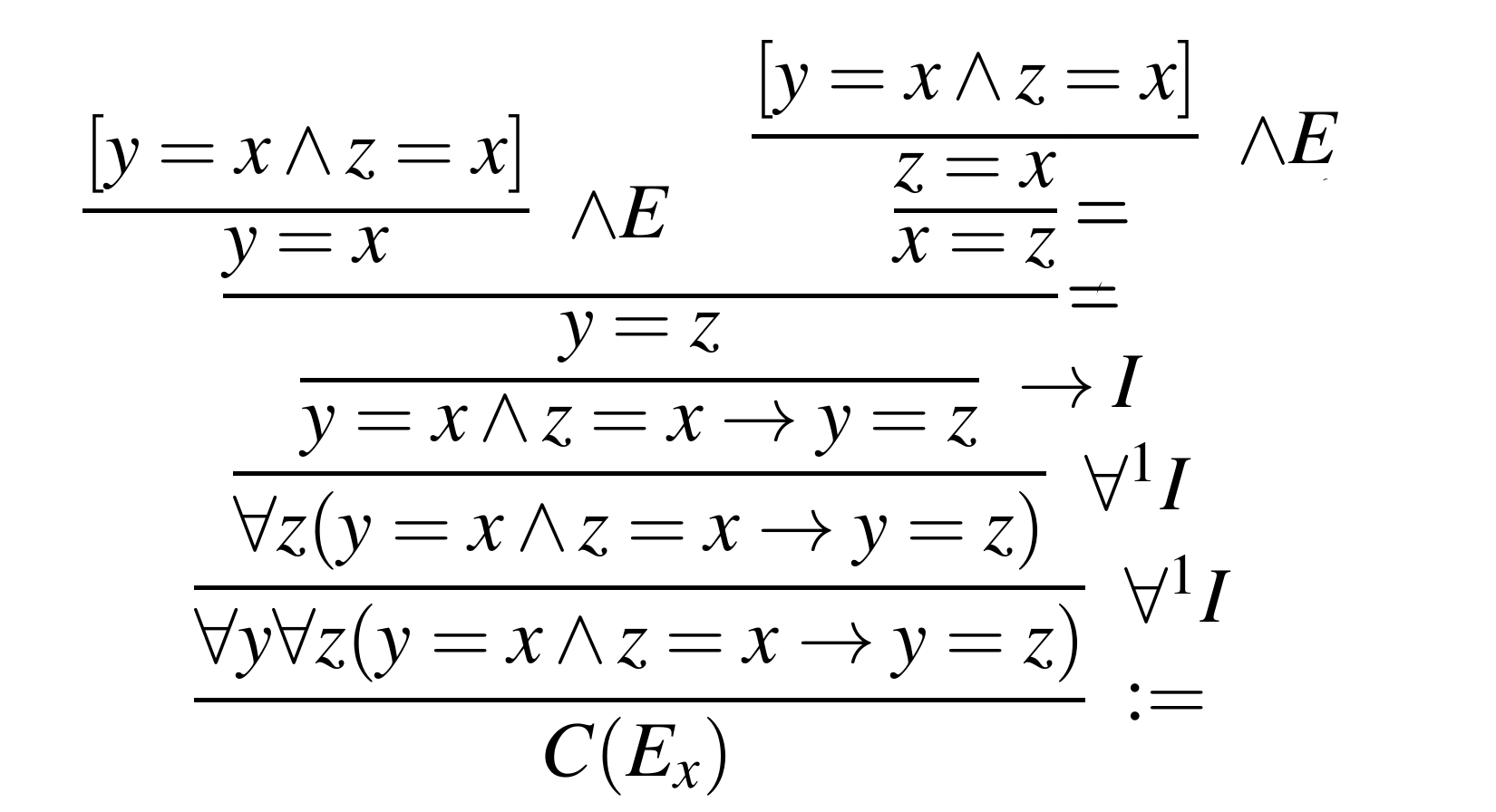}
\end{center} 

\begin{enumerate} 
\item 
\begin{enumerate} 
\item Assuming $\forall x \varphi^{\downarrow}(x)$ one can prove $\forall X(C(X) \limp  \varphi(X))$
\begin{center}
\includegraphics[scale=1.1]{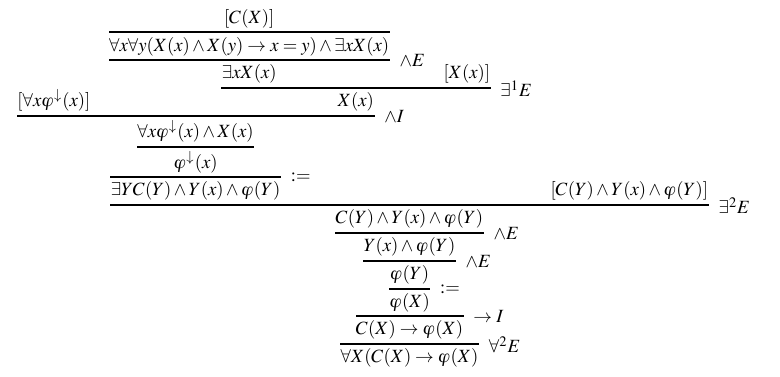} 
\end{center} 
\item Assuming 
$\forall X(C(X) \limp  \varphi(X))$ one can prove $\forall x \varphi^{\downarrow}(x)$. In the proof below, we use $\delta$ the proof that $E_x(\_)=``\_=x"$ is an individual concept. 
\begin{center}
\includegraphics[scale=1.3]{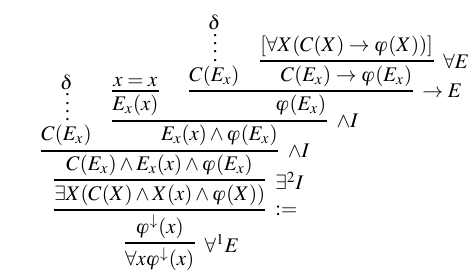} 
\end{center} 
\end{enumerate}
\item 
\begin{enumerate} 
\item Assuming $\forall x \psi(x)$ one can prove $\forall X\left(C(X) \limp  \psi^{\uparrow}(X)\right)$
\begin{center}
\includegraphics[scale=1.14]{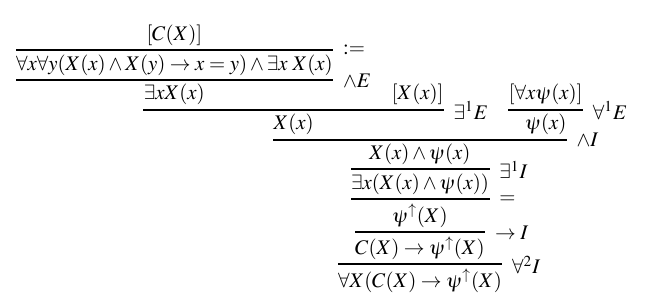} 
\end{center} 
\item Finally assuming $\forall X\left(C(X) \limp  \psi^{\uparrow}(X)\right)$
one can prove: $\forall x \psi(x)$ \nopagebreak 
\begin{center}
\includegraphics[scale=1.4]{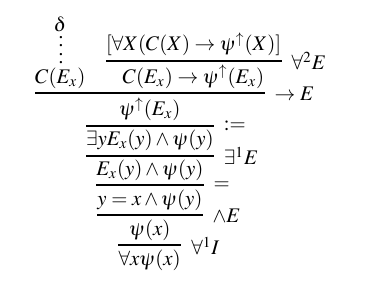} 
\end{center} 
\end{enumerate}
\end{enumerate}
\end{proof}

\subsection[First order existential quantification as existential quantification over non-empty individual concepts]{First order existential quantification as existential quantification over non-empty individual concepts}

As for the universal quantification, we have:

\begin{prop} \label{ICQex}
First order existential quantification and second order quantification over individual concepts are equivalent: 
\begin{enumerate} 
\item when $\phi$ is a property of  individual concepts, one has 
$\exists x \varphi^{\downarrow}(x) \Leftrightarrow \exists X(C(X) \wedge \varphi(X))$ 
where $\varphi^{\downarrow}(x):=\exists X(C(X) \wedge$ $X(x) \wedge \varphi(X))$.
\item  when $\psi$ is a property of  individuals, one has 
$\exists x \psi(x) \Leftrightarrow \exists X\left(C(X) \wedge \psi^{\uparrow}(X)\right)$ where  $\psi^{\uparrow}(X):=\exists x(X(x) \wedge \psi(x))$
\end{enumerate}
\end{prop} 


\begin{proof}
At point 2.a) we will also use the proof $\delta$ of $C(E_x)$ from the proof of proposition \ref{ICQuniv}. 

\begin{enumerate} 
\item $\exists x \varphi^{\downarrow}(x) \Leftrightarrow \exists X(C(X) \wedge \varphi(X))$ 
where $\varphi^{\downarrow}(x):=\exists X(C(X) \wedge$ $X(x) \wedge \varphi(X)$
\begin{enumerate} 
\item Let us prove  $\exists X(C(X) \wedge \varphi(X))$ 
under the assumption $\exists x \varphi^{\downarrow}(x)$.
\begin{center}
\includegraphics[scale=1]{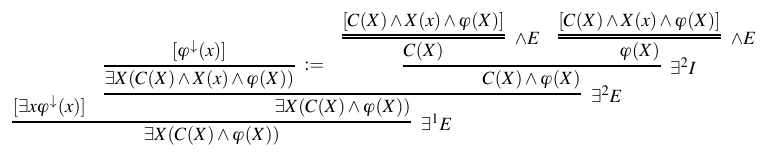} 
\end{center} 
\item Now let us prove  $\exists x \varphi^{\downarrow}(x)$ under the assumption  $\exists X(C(X) \wedge \varphi(X))$. 
We first need $\alpha,\beta,\gamma$ i.e. the three following proofs : 
\medskip 

$$ 
\begin{array}{ll}  
\alpha: & \\ 
& \includegraphics[scale=1.2]{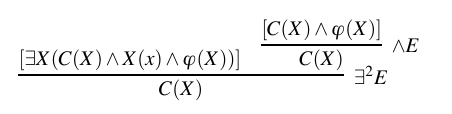}\\ 
\end{array}
$$

$$
\begin{array}{ll} 
\beta:  &  \\ 
& \includegraphics[scale=1.2]{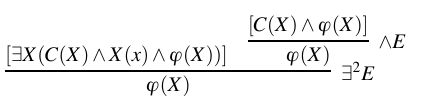}\\
\end{array}
$$

$$
\begin{array}{ll} 
\gamma: & \\ 
& \includegraphics[scale=1.4]{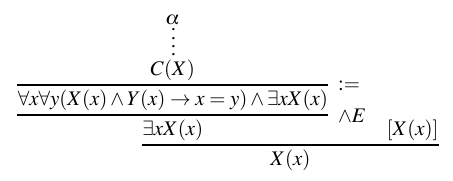}\\
\end{array}$$ 
and the proof we are looking for is: 
\begin{center} 
\includegraphics[scale=1.4]{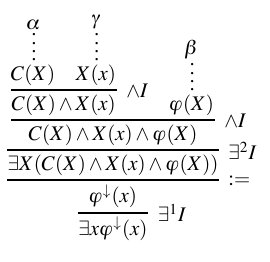}
\end{center} 
\end{enumerate}
\item $\exists x \psi(x) \Leftrightarrow \exists X\left(C(X) \wedge \psi^{\uparrow}(X)\right)$ where  $\psi^{\uparrow}(X):=\exists x(X(x) \wedge \psi(x))$
\begin{enumerate} 
\item Let us prove  $\exists X\left(C(X) \wedge \psi^{\uparrow}(X)\right)$
under the assumption $\exists x \psi(x)$ -- $\delta$ is  the proof of $C(E_x)$ defined in  the proof of proposition \ref{ICQuniv}. 
\begin{center}
\includegraphics[scale=1.4]{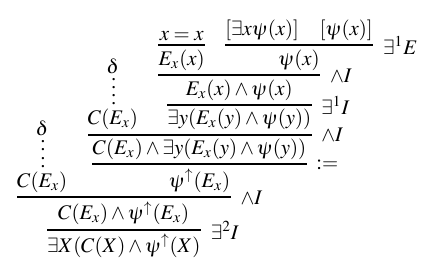} 
\end{center} 
\item 
Let us prove $\exists x \psi(x)$
under the assumption $\exists X\left(C(X) \wedge \psi^{\uparrow}(X)\right)$
\begin{center}
\includegraphics[scale=1.2]{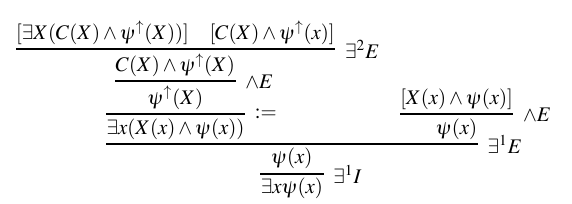} 
\end{center} 
\end{enumerate}
\end{enumerate}
\end{proof}

\subsection{Dealing with possibly empty individual concepts}

When individual concepts are possibly empty the second order formula expressing that $X$ is an individual concept is $C(X)=\forall x \forall y(X x \wedge X y \limp  x=y)$ --- the  $\exists x X(x)$  left out of our initial definition of individual concepts. 

Regarding universal quantification, $\forall X(C(X) \limp  \varphi(X))$ still entails $\forall x \varphi^{\downarrow}(x)$, but  $\forall x \varphi^{\downarrow}(x)$ does not entail  $\forall X(C(X) \limp  \varphi(X))$ anymore. This is logical: when all individual concepts have a property  be they empty or not, all individuals enjoy the corresponding first order property. The converse does not hold: when all individuals enjoy a property, all the non empty concepts enjoy the property, but why should the empty individual concept enjoy this property as well? 

Of course regarding existential quantification, that's the opposite. 
$\exists x \varphi(x)$ entails  $\exists X(C(X) \wedge \varphi(X)))$ but $\exists X(C(X) \wedge \varphi(X))$ does not entail $\exists x \varphi(x)$. When an individual enjoys a property, so does the corresponding individual concept. But when a possibly empty individual concept enjoys a property, it does not entail that an individual enjoys this property, because this individual concept might be an empty individual concept.

So the second order view of usual quantification does not fit in well with possibly empty individual concepts. 

\section
{A reminder on generalised and branching quantifiers}

This reminder mainly relies on the presentation given by Peters and Westerst{\aa}hl \cite{PetersWesterstahl2006quantifiers}, oen may also consult the survey \cite{sep-generalized-quantifiers}. 
Generalised quantifiers, initially introduced by Mostowski \cite{Mostowski1957} and further developed by Lindström\cite{Lindstrom1966} are a generalisation of standard universal and existential quantification. Roughly speaking, generalised quantifiers view quantifiers as relations over relations (or tuples of relations) --- those relations are relations on the domain (a.k.a universe, model) of an interpretation: thus, quantifiers are viewed as second-order concepts. 

\subsection{Generalised quantifiers}

Given $k$ integers  $n_1,..., n_k$, a quantifier $\mathcal{Q}$ of type $\langle n_1,..., n_k \rangle$ can be viewed as a function endowing each domain $M$ with a $k$-ary relation $Q_M$ such that if $(R_1,..., R_k) \in Q_M$, then for all $i$ in  between $1$ and $k$, $R_i$ is a $n_i$-ary relation over elements of $M$. Let us give some examples.

The usual quantifiers $\forall$ and $\exists$ can be then expressed as simple type $\langle 1 \rangle$ quantifiers :
$\exists_M = \{A \subseteq M,\ A \neq \emptyset\}$ and 
$\forall_M = \{A \subseteq M,\ A = M\}$. 
Thus, the existential quantifier is in every domain the unary relation which holds true for all non-empty predicates, and the universal quantifier is the relation which holds true only for the whole domain $M$.

Some generalised quantifiers have an equivalent formulation in usual first-order logic, such as the quantifier ``at least two":
$(\exists_{\geq 2})_M = \{A \subseteq M,\ |A| \geq 2\}$
which can be expressed with the following first-order formula:
$\exists x \exists y x\neq y$
However this is not always the case. Take for example the type $\langle 1,1 \rangle$ quantifier expressing that most $A$ are $B$:
$\textbf{Most}_M(A, B) \Longleftrightarrow |A \cap B| > |A - B|$
which notably cannot be expressed as a first-order formula.

It is worth noting that universal and existential quantification on individual concepts as we presented in Section \ref{individualconcepts} can also be formulated in terms of generalised quantifiers. To say that all (resp. some) individual concepts satisfy a property $\varphi$ is  in fact a second-order statement about the predicates $C$ (``to be an individual concept") and $\varphi$. Hence the second order view of first order quantification that we presented in the previous section can be expressed as generalised quantifiers $\forall_C$ and $\exists_C$ with type $\langle 1,1 \rangle$:
$$
\quad
\forall_C(C, \varphi) \Longleftrightarrow C \subset \varphi
\quad \quad 
\exists_C(C, \varphi) \Longleftrightarrow C \cap \varphi \neq \emptyset 
$$
\subsection{Branching quantifiers}


Among generalised quantifiers, branching quantifiers are of particular interest, both for logic and linguistics. Initially introduced by Henkin \cite{Henkin1961quant} and much later on studied by Hintikka \cite{HS97} — independently of generalised quantifiers — branching quantification is a generalisation of classical quantification that allows the expression of independence  between some  existentially quantified variable and some previously universally quantified variables. This cannot be expressed within usual quantification because quantifiers are supposed to be linearly ordered. The simplest example of such a non-first-order quantifier is the following  Henkin quantifier where, as the notation suggests, $x'$ only depends on $x$, while $y'$, only depends on $y$. 
 
\centerline{
\begin{tikzpicture}[grow=left, sibling distance=20pt,level distance=2.25cm,
edge from parent path={(\tikzparentnode.west) -- (\tikzchildnode.east)}]
\node {$F(x,y,x',y')$}
child {node {$\forall x \exists x'$}}
child {node {$\forall y \exists y'$}};
\end{tikzpicture}
}

As proven by Ehrenfeucht (in Henkin \cite{Henkin1961quant}), this construction has no first-order equivalent. Notably, it cannot be expressed with a linear quantifier prefix such as $\forall x \exists x' \forall y \exists y'$ or $\forall x \forall y \exists x' \exists y'$, since there would be unwanted dependencies between $x'$ and $y$, and $y'$ and $x$.

Although not initially introduced as such, branching quantifiers can in fact be seen as specific generalised quantifiers. Indeed, the (in)dependencies between variables can be expressed using Skolem functions, e.g. the Henkin quantifier above can be written as follows: 
$$\exists f \exists g \forall x \forall y\ F(x, f(x), y, g(y))$$
This in turn allows us to translate it as a generalised quantifier, for example here as the type $\langle 4 \rangle$ quantifier:\quad 
$$H_M = \{ R \subseteq M^4\ |\ \exists f \exists g \forall x \forall y\ (x,f(x),y,g(y)) \in R\}$$

\section{Second-order proof rules for branching quantifiers}


\paragraph{Branching quantifiers as second-order formulae}

In this part, we focus on the expression of branching quantifiers as second-order constructions. Such quantifiers can be quite complex, so we limit ourselves to studying the simplest branching quantifier. Our main object of study is the typical branching constructions found in natural language in the so-called Hintikka sentences, such as :
\medskip 

\noindent
    \textbf{(H)}\quad  \textit{A member of each team and a member of each board of directors know each other}
\medskip 

The branching-quantifier reading of the above English sentence can be formulated within second-order logic:\footnote{There is also a first-order reading of this sentence, which the 'each other' (perhaps) makes less perceptible, and which can be expressed within first-order logic: $[\forall x \exists x' \forall y \exists y'\quad T(x) \wedge B(y) \limp  M(x, x') \wedge M(y, y') \wedge K(x', y')] \land 
[\forall y \exists y'\forall x \exists x' \quad T(x) \wedge B(y) \limp  M(x, x') \wedge M(y, y') \wedge K(x', y')]$. According to Szymanik \cite{Szymanik2016}, in two-thirds of cases, the first first-order reading is preferred to the branching quantifier reading.
}
\begin{center}
\begin{tikzpicture}[grow=left, sibling distance=20pt,level distance=5cm,
edge from parent path={(\tikzparentnode.west) -- (\tikzchildnode.east)}]
\node {$T(x) \wedge B(y) \limp  M(x, x') \wedge M(y, y') \wedge K(x', y')$}
child {node {$\forall x \exists x'$}}
child {node {$\forall y \exists y'$}};
\end{tikzpicture}
\end{center}

As mentioned earlier, this formula can be expressed as a second order formula with existential quantification over functions:

$$(\Hfun): \exists f \exists g \forall x \forall y\ T(x) \wedge B(y) \limp  K(f(x), g(y))$$

\paragraph{Natural deduction rules with binary predicates}

This formulation of the Henkin quantifier as a second-order formula with quantification over functions is however not fully satisfactory, for it actually provides a stronger effect than needed: defining $f$ and $g$ as functions implies the unicity of $f(x)$ and $g(y)$ for any given $x$ and $y$, while the original formula with a branching quantifier only requires that there exists one (possibly more) $x'$ for each $x$ and $y'$ for each $y$. Thus $f$ and $g$ need not be functions, but only need be non-empty binary predicates --- as always with Skolem functions, the choice of $f(x)$ for each $x$ is part of the interpretation of the function symbol.

Therefore, we propose another second-order representation of this reading of the sentence using quantification over predicates instead of quantification over functions:

\begin{equation*}
    \begin{split}
    (\Hpred): &\exists F \exists G [\forall x \exists x' T(x)\limp F(x,x')] \wedge [\forall y \exists y' B(y)\limp G(y,y')]\\ &\wedge [\forall x \forall x' \forall y \forall y' T(x) \wedge B(y) \wedge F(x, x') \wedge G(y, y') \limp  K(x', y')]
 \end{split}
\end{equation*}

This formula simply replaces each of the two functions $f$ and $g$ of $(Hfun)$ above with the binary predicates $F$ and $G$. These two predicates act intuitively as relations that select suitable $x'$ and $y'$, since all we need to ensure is that whenever $x'$ is a valid representative for $x$ (and $y'$ for $y$), then $x'$ and $y'$ know each other. The binary predicates $F$ and $G$ are required to relate each possible value of their first argument which ought to be in the proper set/predicate ($T$ for $x$, $B$ for $y$) to at least one value of their second argument.\footnote{Similarly, we could ask that $x'$ and $y'$ are in the proper set/predicate ($T$ for $x'$, $B$ for $y'$) but it is less important. Thus we do not add this precision, which is not needed --- unlike the restriction to $x$ and $y$ --- and makes the formulae, which are already long enough, considerably longer:     $\Hpred': \exists F \exists G [\forall x \exists x' T(x)\limp F(x,x') \land T(x')] \wedge [\forall y \exists y' B(y)\limp G(y,y')\land B(y')] \wedge [\forall x \forall x' \forall y \forall y' T(x) \wedge T(x') \wedge B(y) \wedge B(y') \wedge F(x, x') \wedge G(y, y') \limp  K(x', y')]$} There is nevertheless a difference between using function as in $(\Hfun)$ and $(\Hpred)$: in $(\Hpred)$ there may well exist several values of $x'$ (resp. $y'$) such that $F(x,x')$ (resp. $G(y,y')$), there is no need to chose one as opposed to  $(\Hfun)$, where one is explicitly chosen.

The natural deduction rules for second-order logic from section \ref{proofrules} give us the introduction and elimination rules for the branching Henkin quantifier. For the sake of readability, let us write from here on:
$$
\Phi(F,G,x,x',y,y') = T(x) \wedge B(y) \wedge F(x, x') \wedge G(y, y') \rightarrow K(x', y')
$$
and 
\begin{equation*}
    \begin{split}
\Psi(F,G) =\ &[\forall x \exists x' T(x)\limp F(x,x')] \wedge [\forall y \exists y' B(y)\limp G(y,y')]\\ &\wedge [\forall x, \forall x' \forall y \forall y' \Phi(F,G,x,x',y,y')] 
 \end{split}
\end{equation*}

The introduction rule is quite straightforward:
$$
\begin{prooftree}
    \varphi(x, t)
    \qquad
    \psi(y, u)
    \qquad
    \Phi(\varphi,\psi,x,x',y,y')
    \justifies
    \exists F \exists G [\forall x \exists x' T(x)\limp F(x,x')] \wedge [\forall y \exists y' B(y)\limp G(y,y')] \wedge [\forall x, \forall x' \forall y \forall y' \Phi(F,G,x,x',y,y')] 
    \using H_I
\end{prooftree}
$$
where $\varphi$ is a formula with free variables exactly $x, x'$, $\psi$ a formula with free variables exactly $y, y'$, and $t,u$ are terms.

The elimination rule, however, is more complicated due to the use of two second-order eliminations of $\exists^2$:
$$
\begin{prooftree}
            \exists F \exists G \Psi(F,G)
                \quad
        \begin{prooftree}
            [\exists G \Psi(A,G)]^{(2)}
                \quad
                \hspace*{0em}
            \[
            [\Psi(A,B)]^{(1)}
            \leadsto
            \varphi
            \]
            \quad
            \hspace*{0em}
            \justifies
            \varphi
            \using \exists^{2}E^{(1)}
        \end{prooftree}
        \justifies
        \varphi
        \using \exists^{2}E^{(2)}
\end{prooftree}
$$
in which $A$ and $B$ must not appear free in $\varphi$.

Let us write as above  $H(x,x',y,y')$ the branching quantifier that binds the two universally quantified variables $x,y$ and the  two existentially quantified variables $x',y'$ e.g. the example above may be written 
$H(x,x',y,y')\ \Phi(F,G,x,x',y,y')$. 

Let $\mathcal{H}$ be the set of closed formulae  that can be written with this quantifier and the two usual first order quantifiers $\exists x$ and $\forall y$. 

In the near future, we intend to determine whether our direct rules can be used to derive all the formulae in $\mathcal{H}$ that can be derived with the usual rules for second and first quantifiers. It seems plausible to us, because cut-elimination holds for second order logic. \cite{HetzlLeitschWelller2011apal}   However, we are not yet fully certain, and this is one aspect that we intend to clarify in the near future.

\section{Conclusion}

This ongoing  work deals with the proof rules  of extensions of first order quantification viewed as second order logic constructions which do not need the full expressive power (and complexity) of second order logic.  

We first described usual first order quantifiers within the second order logic as quantification over individual concepts and obtained that this description works provided individual concepts are asked to be non empty  --- unsurprisingly it does not work without this restriction. 

Thereafter we focused on the non first order reading of the simplest branching quantifier that one finds in sentences like:
\textit{``A member of each team and a member of each board of directors know each other"}. 
Regarding this (reading of this) quantifier,  we proposed a definition of it within second order logic and provided direct introduction and elimination rules for this complex quantifier which is often only described in model-theoretic terms. 
Later on we shall prove that the complete set of second order rules does not derive more sentences with those connectives than our direct rules. 

By way of conclusion, let us mention a prospect that has opened up to us recently.  The epsilon calculus,\cite{epsilonURL,CPR2017jla}  may express readings  that are close to branching quantifier readings,  with epsilon formulas that have no equivalent in first or higher order logic. Indeed, the subnectors epsilon and tau express quantification with some scope ambiguities (under-specification). However, it is presently too early to say anything definite about this question. 

\medskip

\noindent\textit{\textbf{Last minute remark:} While we were revising
this article we realised that Matthias Baaz and Anela Lolic
\cite{BL2021} proposed an analytic calculus for Henkin quantifiers. In their paper, Henkin quantifiers are viewed as the existence of functions  (i.e. as a particular form of second order quantification), and they proved cut-elimination for this calculus.  It is too late for this paper of ours to examine how their account of Henkin quantifiers differs from ours, but this will be
be our first aim in continuing our work. A first remark is that $(1): \exists F \forall u \exists v F(u,v)$ seems slightly weaker than $(2): \exists f F(u,f(u))$: in order to derive $(2)$ from $1$, it is necessary to utilise some form of the axiom of choice. Another difference, a very small one actually, is that our rules are natural deduction rules (introduction/elimination rules) and not sequent calculus: sequent-calculus right-rules correspond to natural-deduction introduction-rules but sequent-calculus  left-rules do not correspond to natural-deduction elimination-rules.}

\begin{acknowledgments}
We would like to express our gratitude to the anonymous reviewers of ARQNL2024 for their valuable feedback, which has been instrumental in helping us to enhance this article. Furthermore, we would like to extend our gratitude to the programme chairs who accepted some late modifications.
\end{acknowledgments}

\newpage 
\bibliography{bigbiblio}

\end{document}